\documentclass[10pt]{amsart}

\usepackage{amscd,amsmath,amsfonts,amsthm,amssymb}
\usepackage{enumitem}
\usepackage{graphicx}
\usepackage{xcolor}
\usepackage[all]{xy}
\usepackage[top=25truemm,bottom=25truemm,left=15truemm,right=15truemm]{geometry}

\usepackage[colorlinks=true,citecolor=blue,linkcolor=magenta,urlcolor=green,bookmarksnumbered=true,pdfborder={1 0 0},bookmarkstype=toc]{hyperref}
\usepackage[nameinlink,noabbrev]{cleveref}

\newcommand\bcdot{\ensuremath{%
  \mathchoice%
   {\mskip\thinmuskip\lower0.2ex\hbox{\scalebox{1.5}{$\cdot$}}\mskip\thinmuskip}}%
   {\mskip\thinmuskip\lower0.2ex\hbox{\scalebox{1.5}{$\cdot$}}\mskip\thinmuskip}%
   {\lower0.3ex\hbox{\scalebox{1.2}{$\cdot$}}}%
   {\lower0.3ex\hbox{\scalebox{1.2}{$\cdot$}}}%
   }

\theoremstyle{definition}
\newtheorem{thm}{Theorem}[section]
\newtheorem{dfn}[thm]{Definition}
\newtheorem{lem}[thm]{Lemma}
\newtheorem{pro}[thm]{Proposition}

\newtheorem{rmk}[thm]{Remark}
\newtheorem{nota}[thm]{Notation}
\newtheorem{const}[thm]{Construction}

\crefname{thm}{Theorem}{Theorems}
\crefname{lem}{Lemma}{Lemmas}
\crefname{pro}{Proposition}{Propositions}
\crefname{cor}{Corollary}{Corollaries}
\crefname{dfn}{Definition}{Definitions}
\crefname{exa}{Example}{Examples}
\crefname{rmk}{Remark}{Remarks}
\crefname{nota}{Notation}{Notations}
\crefname{const}{Construction}{Constructions}
\crefname{section}{Section}{Sections}

\newcommand{\exend}{\unskip\nobreak\hfill$\blacklozenge$}

\newcommand{\A}{\mathbb{A}}
\newcommand{\C}{\mathbb{C}}
\newcommand{\D}{\mathcal{D}}

\renewcommand{\L}{\mathbf{L}}
\newcommand{\M}{\mathcal{M}}

\renewcommand{\O}{\mathcal{O}}

\newcommand{\R}{\mathbf{R}}
\renewcommand{\S}{\mathcal{S}}
\newcommand{\Z}{\mathbb{Z}}

\renewcommand{\phi}{\varphi}
\renewcommand{\epsilon}{\varepsilon}

\newcommand{\Hom}{\mathop{\mathrm{Hom}}\nolimits}
\newcommand{\iHom}{\underline{\mathop{\mathrm{Hom}}\nolimits}}

\newcommand{\im}{\mathrm{im}}
\newcommand{\coim}{\mathrm{coim}}
\newcommand{\fib}{\mathrm{fib}}
\newcommand{\cofib}{\mathrm{cofib}}

\newcommand{\ilim}[1][]{\mathop{\varinjlim}\limits_{#1}}
\newcommand{\plim}[1][]{\mathop{\varprojlim}\limits_{#1}}
\renewcommand{\lim}[1][]{\mathop{\mathrm{lim}}\limits_{#1}}
\newcommand{\colim}[1][]{\mathop{\mathrm{colim}}\limits_{#1}}

\newcommand{\op}{\text{op}}

\newcommand{\an}{\text{an}}

\newcommand{\alg}{\text{alg}}
\newcommand{\ad}{\text{ad}}

\newcommand{\Mod}{\mathrm{Mod}}
\newcommand{\BMod}{\mathrm{BMod}}

\newcommand{\Ind}{\mathrm{Ind}}
\newcommand{\Idem}{\mathrm{Idem}}

\newcommand{\Spec}{\mathrm{Spec}}
\newcommand{\Spa}{\mathrm{Spa}}

\begin{document}

\title[Analytification for Complex Geometry Revisited]{Analytification for Complex Geometry Revisited}

\author[Y. Yamada]{Yuto Yamada}
\address{Department of Mathematics, Institute of Science Tokyo, 2-12-1 Ookayama, Meguro, Tokyo 152-8551}
\email{yamada.y.f243@m.isct.ac.jp}
\begin{abstract}
In this paper, we provide a new framework to interpret complex geometry, inspired by Bambozzi--Chiarellotto---Vanni's work on tempered cohomology. We define several ind-Banach $\C$-algebras of overconvergent and holomorphic power series, and verify some desirable properties to endow Berkovich spaces with “analytic" structures in analogy with Clausen--Scholze's condensed formalism. As an application, we obtain an abstract GAGA-type comparison in this setting.
\end{abstract}


\date{\today}

\maketitle

\tableofcontents

\section{Introduction}

\subsection*{Historical background}

In the 1930s, de Rham defined the \emph{analytic} de Rham cohomology of complex manifolds via differential forms, and compared with the singular cohomology giving many topological invariants  ({\cite{dR}}). Building on this, in the 1950s, Serre proved an important relationship between the analytic and algebraic viewpoints ({\cite{S}}), which is known as GAGA. In the 1960s, Grothendieck defined the \emph{algebraic} de Rham cohomology of algebraic varieties over an arbitrary field (not only over $\C$), and proved a comparison between it and the analytic de Rham cohomology above via analytification as in {\cite{S}}. These classical works show that one can move flexibly between analytic and algebraic techniques.

On the other hand, the theory of Banach algebras (or the Gelfand--Naimark theory of $C^\ast$-algebras) provided a bridge between functional analysis and geometry in the 1940s ({\cite{Ge}} and {\cite{GN}}). In the 1990s, Berkovich developed the theory of \emph{Berkovich spaces} (i.e., the spaces of multiplicative seminorms), which is based on not only “discrete" rings, but also Banach rings ({\cite{Be}}). This theory gives richer (topological) information than the classical spaces (i.e., spectral spaces). This includes, as an important special case, the theory of \emph{rigid analytic spaces} developed by Tate ({\cite{T}}).

In the theory of ind-Banach rings or bornological rings, Bambozzi, Ben-Bassat and some collaborators recently provided a framework where both Archimedean and non-Archimedean worlds can be treated in a general categorical setting (for example, by \cite{BBM}).

More recently, Clausen--Scholze introduced the notions of \emph{condensed mathematics and analytic rings/stacks}, which have the potential to unify some theories such as (non-)Archimedean theory, complex geometry, topological theory and so on. In particular, they allow us to re-interpret complex geometry in {\cite{CC}} via \emph{liquid} structures. The idea was summarized as follows: to equip the underlying topological spaces of complex manifolds with “analytic" (liquid) structures via \emph{categorified locales}, which were originally developed by Balmer--Krause--Stevenson ({\cite{BKS}}). Also, on the non-Archimedean side, Bambozzi--Chiarellotto--Vanni developed a \emph{“tempered”} analytic framework, which provides a new comparison for crystalline cohomology. They also used the geometry of categorified locales over ind-Banach non-archimedean modules as above.

\subsection*{Results}

First, we will define the ring $\C\{|T|\leq1\}$ which corresponds to the closed unit disc in the framework of ind-Banach $\C$-modules. Also, we will construct other rings $\C\{|T|\geq1\},\C\{|T|<1\}$, and prove that they satisfy some axiomatic properties. As a simple consequence, we prove a type of GAGA axiomatically:

\begin{thm}[\cref{thm:GAGA}]
For a proper variety $X$ over $\C$, we obtain an equivalence $C^\an(X)\simeq C^\alg(X)$.
\end{thm}

\subsection*{Structure of this paper}

In \cref{section:Categorical Preliminaries}, we will recall techniques from the theory of locales to deal with a large $\infty$-category $\mathrm{Sym}$ according to \cite[Lecture V, V\hspace{-1pt}I\hspace{-1pt}I]{CC}. Also, we will define the notion of categorical locales to equip Berkovich spaces with a kind of “complex structure".

In \cref{section:Recollection on (ind-)Banach algebras}, we will review the theory of (ind-)Banach modules/algebras, and some categorical properties of the (derived) category of ind-Banach modules.

In \cref{section:The Analytic Line}, we will define some rings which are analogues of holomorphic functions in the theory of complex analysis. Also, we will prove some properties (for example, idempotency) by explicit calculations to use the framework of \cref{section:Categorical Preliminaries}. As an application of this framework, we prove “abstract" GAGA by using a categorical formalism in \cref{section:Abstract GAGA}.

In \cref{section:Abstract GAGA}, we will summarize an axiomatic approach for the so-called GAGA according to {\cite[Lecture  V\hspace{-1pt}I,  V\hspace{-1pt}I\hspace{-1pt}I]{CC}}.

\subsection*{Notation}

In this paper, we consider the complex number field $\C$ as a Banach field with the usual Euclidean norm $|-|_\C$. Basically, we work in (stable) $\infty$-categories. 

\subsection*{Acknowledgements}

The author would like to thank Makoto Enokizono for helpful discussions on ind-Banach algebras. This work was supported by JST SPRING, Japan Grant Number JPMJSP2180.

\section{Categorical Preliminaries}\label{section:Categorical Preliminaries}

In this section, we review the notion of \cite[Lecture V, V\hspace{-1pt}I\hspace{-1pt}I]{CC}.

\begin{nota}
We fix a presentably closed symmetric monoidal stable $\infty$-category $C$, and let $\Idem(C)$ denote the set of idempotent commutative algebra objects (i.e., algebra objects with a map $1\to A$ which induces an isomorphism $A\simeq 1\otimes A\xrightarrow{\sim}A\otimes A$) of $C$. For any idempotent commutative algebra object $A$ of $C$, we can define the $\infty$-category of $A$-modules as follows:
\[
\Mod_A(C):=\{X\in C\mid X\otimes A\simeq X\}.
\]
Note that some (higher) coherence conditions (for $A$-module structure) can be checked by this datum $X\otimes A\simeq X$ only.
\exend
\end{nota}

\begin{const}[locales]
We can define a poset structure on $\Idem(C)$ as $A\leq A'$ if and only if there is a map $A'\to A$ which commutes with their structure maps by {\cite[Construction 5.2]{CC}}. A \emph{locale $\S(C)$ (associated with $C$)} is defined by the following data (infinite unions and intersections):
\begin{enumerate}
    \item $Z\cap Z'$ corresponds to $A\otimes A'$.
    \item $Z\subset Z'$ if and only if $A\otimes A'=A$.
    \item $Z\cup Z'=Z\sqcup_{Z\cap Z'}Z'$ corresponds to the fiber of $A\oplus A'\to A\otimes A'$ (equivalent to $A\times_{A\otimes A'}A'$).
    \item $\bigcap_iZ_i$ corresponds to $\colim[i]A_i$.
\end{enumerate}
where each closed subset $Z_{(i)}^{(')}$ of $\S(C)$ denotes the corresponding idempotent commutative algebra object $A_{(i)}^{(')}$. Indeed, this construction defines locales (i.e., satisfies the distribution law) by {\cite[Proposition 5.3]{CC}} (see \cite{A} for more discussion).
\exend
\end{const}

\begin{dfn}[open/closed inclusions]\label{dfn:open/closed inclusions}
We let $\mathrm{Sym}$ denote the (large) $\infty$-category of closed symmetric monoidal stable $\infty$-categories with colimits. Let $f^\ast:D\to E$ be a map of $\mathrm{Sym}^\op$.
\begin{enumerate}
    \item The map $f^\ast$ is a \emph{closed inclusion} if it admits a colimit-preserving fully faithful right adjoint functor $f_\ast:E\to D$ such that $f_\ast$ satisfies \emph{the projection formula}:
    \[
    d\otimes f_\ast e\xrightarrow{\sim}f_\ast(f^\ast d\otimes e)\quad\text{for all $d\in D$ and $e\in E$}.
    \]
    Equivalently, if there exists an idempotent algebra $A$ ($=f_\ast 1_E$) of $\Idem(D)$ such that $1_E\xrightarrow{\sim}f^\ast A$ and $\Mod_A(D)\xrightarrow{\sim}E$ (see also {\cite[Lemma 6.4(1) \& Proposition 6.5(1)]{CC}}).
    \item The map $f^\ast$ is an \emph{open inclusion} if it admits a (colimit-preserving) fully faithful left adjoint functor $f_\natural:E\to D$ such that $f_\natural$ satisfies \emph{the projection formula}:
    \[
    f_\natural(f^\ast d\otimes e)\xrightarrow{\sim}d\otimes f_\natural e\quad\text{for all $d\in D$ and $e\in E$}.
    \]
    Equivalently, if there exists an idempotent algebra $A$ ($=\cofib(f_\natural1_E\to1_D)$) of $\Idem(D)$ such that $f^\ast A=0$ and $D/\Mod_A(D)\xrightarrow{\sim}E$ (see also {\cite[Lemma 6.4(2) \& Proposition 6.5(2)]{CC}}).
\end{enumerate}
Note that these (equivalent) definitions give the compatibility with base-change (by the last statement of {\cite[Proposition 6.5]{CC}}) and the geometric properties in the sense of \cite[Definition 2.1.1]{HM} (by {\cite[Corollary 6.6]{CC}} and {\cite[Lemma 2.1.5]{HM}}).
\exend
\end{dfn}

\begin{const}[structure sheaves]\label{const:structure sheaves}
For a closed subset $Z$ of $\S(C)$ corresponding to an idempotent algebra $A\in\Idem(C)$, we define the following notions:
\begin{enumerate}
    \item We define $C(Z)$ as the full $\infty$-subcategory $\Mod_A(C)$ of $C$.
    \item The inclusion $C(Z)\subset C$ is denoted by $i_{Z,\ast}$.
\end{enumerate}
In fact, $i_{Z,\ast}$ admits a left adjoint functor $i_Z^\ast:C\to C(Z)$ given by $X\longmapsto X\otimes A$, and a right adjoint functor $i_Z^!:C\to C(Z)$ given by $X\longmapsto\iHom(A,X)$. By dual discussion, we formally obtain the “complementary open subset" $U$ of $Z$, and define the following notions:
\begin{enumerate}
    \item We define $C(U)$ as the localization $\infty$-category $C/\Mod_A(C)=C/C(Z)$ of $C$.
    \item The localization $C\to C(U)$ is denoted by $j_U^\ast$.
\end{enumerate}
In fact, $j_U^\ast$ admits a left adjoint functor $j_{U,!}:C(U)\to C$ satisfying $j_{U,!}j_U^\ast X=\fib(X\to X\otimes A)$, and a right adjoint functor $j_{U,\ast}:C(U)\to C$ satisfying $j_{U,\ast}j_U^\ast X=\iHom(\fib(1\to A),X)$.

Moreover, these constructions give the following distinguished triangles:
\[
j_{U,!}j_U^\ast X\to X\to i_{Z,\ast}i_Z^\ast X\xrightarrow{+1},\quad i_{Z,\ast}i_Z^!X\to X\to j_{U,\ast}j_U^\ast X\xrightarrow{+1}.
\]
Furthermore, the functor $U\longmapsto C(U)$ defines a sheaf (as the usual one on locales) by {\cite[Proposition 5.5]{CC}}. In particular, the sheaf $U\longmapsto j_{U,\ast}j_U^\ast1_C=\iHom(\fib(1\to A),1)$ is called the \emph{structure sheaf}.
\exend
\end{const}

\begin{rmk}[open/closed descent]\label{rmk:open/closed descent}
Moreover, we can obtain the following formal descent in $\mathrm{Sym}^\op$ with respect to a Grothendieck topology defined by open inclusions (\cref{dfn:open/closed inclusions}) by {\cite[Theorem 6.7]{CC}}:
\begin{enumerate}
    \item The identity functor $(\mathrm{Sym}^\op)^\op\to\mathrm{Sym}$ defines a sheaf with respect to this Grothendieck topology above.
    \item The poset of open/closed inclusions satisfies descent with respect to this Grothendieck topology above.
\end{enumerate}
\exend
\end{rmk}

\begin{dfn}[{\cite[Definition 7.1]{CC}}, categorical locales]\label{dfn:categorical locales}
A \emph{categorified locale} is a triple $(X,C,f:\S(C)\to X)$ where $X$ is a locale, $C$ is a presentably closed symmetric monoidal stable $\infty$-category, and $f$ is a map of locales. If there is no confusion, then we abbreviate $(X,C)$. Note that the assignment $U\longmapsto C(f^{-1}(U))$ on $X$ defines a sheaf with values in $\mathrm{Sym}$ (by \cref{rmk:open/closed descent}).
\exend
\end{dfn}

\section{Recollection on (ind-)Banach algebras}\label{section:Recollection on (ind-)Banach algebras}

In this section, we define the derived $\infty$-category of Banach $\C$-algebras, and construct complex analytic spaces. First, we recall the properties for the ind-category of Banach $\C$-modules following {\cite[Section 3,4]{BCV}}.

\begin{dfn}[some notations]\label{dfn:some notations}
We define the following (perhaps, well-known) notions:
\begin{enumerate}
    \item ({\cite[Definition 3.2]{BBK}}) In a pre-abelian category (i.e., an additive category with kernels and cokernels), a map is \emph{strict} if there is a canonical isomorphism $\coim(f)\xrightarrow{\sim}\im(f)$.
    \item ({\cite[Definition 3.5]{BBK}}) A category $C$ is \emph{quasi-abelian} if $C$ is an additive category with kernels and cokernels, and the class of short strictly exact sequences makes $C$ an exact category.
    \item ({\cite[Definition 3.9, 3.10]{BBK}}) A quasi-abelian category $C$ has \emph{enough projectives} if any object $X\in C$ admits a strict epimorphism from a projective object. Note that the notion of \emph{projective} is defined if the functor $\Hom(X,-):C\to\mathrm{Ab}$ sends strict epimorphisms to epimorphisms.
    \item ({\cite[Definition 3.12, 3.13]{BBK}}) A quasi-abelian category $C$ with colimits is \emph{elementary} if $C$ admits a (small) family of compact projective objects satisfying the following condition: any object $X\in C$ admits a strict epimorphism $\bigoplus_{i\in I}X_i\to X$ where each $X_i$ lies in the generating family. Note that the notion of \emph{compact} is defined if the functor $\Hom(X,-):C\to\mathrm{Ab}$ commutes with filtered colimits.
    \item ({\cite[Definition 3.16, Lemma 3.17]{BBK}}) The \emph{ind-category} $\Ind(C)$ of a category $C$ is defined as the full subcategory of $\mathrm{PShv}(C,\mathrm{Set})$ spanned by filtered colimits of representable functors. Note that if $C$ is a small closed symmetric monoidal quasi-abelian category such that it has enough projectives, then $\Ind(C)$ is a closed symmetric monoidal elementary quasi-abelian category with limits and colimits ({\cite[Proposition 3.20]{BBK}}).
\end{enumerate}
Let $C$ be a small closed symmetric monoidal quasi-abelian category such that it has enough projectives, and the tensor product of two projectives is projective. The \emph{derived $\infty$-category} $\D(\Ind(C))$ is defined by a closed symmetric monoidal stable $\infty$-category represented by the category of cochain complexes of $\Ind(C)$ (discussed in {\cite[Subsection 3.4]{BCV}}). Note that the homotopy category of $\D(\Ind(C))$ is the derived category of a quasi-abelian category $\Ind(C)$ defined in {\cite[Definition 3.6]{BBK}}.
\exend
\end{dfn}

As an important example satisfying the conditions above, we need the notion of Banach $\C$-modules.

\begin{dfn}[{{\cite[Definition 3.28]{BBK}}}, Banach $\C$-modules]
A $\C$-module $M$ is \emph{Banach} if $M$ admits a map $|-|_M:M\to\mathbb{R}_{\geq0}$ satisfying the following conditions:
\begin{enumerate}
    \item We have $|m|_M=0$ if and only if $m=0$. 
    \item For any elements $m,n\in M$, we have $|m+n|_M\leq|m|_M+|n|_M$.
    \item For any elements $m\in M$, $x\in\C$, we have $|xm|_M\leq|x|_\C|m|_M$.
    \item The module $M$ is complete with respect to the induced topology by $|-|_M$.
\end{enumerate}
We let $\BMod_\C$ denote the category of Banach $\C$-modules with bounded maps (i.e., maps $f:M\to N$ of $\C$-modules satisfying that there is some constant $C>0$ such that $|f(m)|_N\leq C|m|_M$ for any element $m\in M$).
\exend
\end{dfn}

Thanks to {\cite[Lemma 3.68]{BBK}}, we can use ind-Banach $\C$-modules properly.

\begin{dfn}[derived $\infty$-category: local]\label{dfn:local}
The derived $\infty$-category $\widehat{\D}(\C)$ of \emph{ind-Banach $\C$-modules} is defined as the closed symmetric monoidal stable $\infty$-category $(\D(\Ind(\BMod_\C)),\widehat{\otimes}_\C^\L)$ (see \cref{dfn:some notations}). Let $A$ be a commutative monoid in $\widehat{\D}(\C)$ (for example, we can take an (ind-)Banach $\C$-algebra $A$ in the sense of {\cite[Definition 3.27]{BBK}}). We also define the derived $\infty$-category $\widehat{\D}(A)$ of \emph{ind-Banach $A$-modules} as the closed symmetric monoidal stable $\infty$-category $(\Mod_A(\widehat{\D}(\C)),\widehat{\otimes}_A^\L)$.
\exend
\end{dfn}

Therefore, we can consider categorified locales over $(\ast,\widehat{\D}(\C))$. Before our definition of complex analytic spaces, we restate a useful criterion.

\begin{rmk}[idempotency]\label{rmk:idempotency}
In $\widehat{\D}(\C)$, a map $A\to B$ in $\mathrm{Comm}(\widehat{\D}(\C))$ (where $\mathrm{Comm}(\widehat{\D}(\C))$ denotes the $\infty$-category of commutative algebra objects in $\widehat{\D}(\C)$) is a \emph{homotopy epimorphism} if the canonical map $B\widehat{\otimes}_A^\L B\to B$ is an isomorphism ({\cite[Lemma 2.14]{BBM}}). Also, a pair $(A,x)$ (where $A$ is a commutative algebra object in $\widehat{\D}(\C)$ and $x$ is an element of $A$) is \emph{strict} if $A$ is flat, the map $(x\otimes 1-1\otimes x):A\widehat{\otimes}_\C^\L A\to A\widehat{\otimes}_\C^\L A$ is strict (as a map in cochain complexes) and its cokernel is isomorphic to $A$ itself ({\cite[Definition 5.2]{BBM}}). With these notions, there are some properties for idempotency.
\begin{enumerate}
    \item  Homotopy epimorphisms are stable under derived base change by {\cite[Proposition 3.4]{BBM}}, under filtered colimits by {\cite[Proposition 3.5]{BBM}}, and under tensor products by {\cite[Proposition 3.6]{BBM}} under the flatness assumption.
    \item ({\cite[Theorem 5.3]{BBM}}) For a map $f:A\to B$ in $\mathrm{Comm}(\widehat{\D}(\C))$, if two pairs $(A,x)$ and $(B,f(x))$ are strict, then we obtain $B\widehat{\otimes}_A^\L B\simeq B$ (i.e. $A\to B$ is a homotopy epimorphism).
\end{enumerate}
By {\cite[Lemma 4.6]{BBM}}, we have a pair $(\C[T],T)$ as an example of strict pairs. Note that we consider $\C[T]$ as an ind-Banach $\C$-algebra (see below \cref{rmk:Berkovich realization for complex geometry}), and it is equal to the one-variable polynomial ring $\mathrm{Sym}(\C)$ in the sense of \cite[Subsection 4.1]{BBM}. If we have a strict pair $(A,x)$ with $\C[T]\xrightarrow{T\longmapsto x}A$, then by (2), the ring $A$ is idempotent over $\C[T]$,  and by (1), the $n$-th fold $A^{\widehat{\otimes}_{\C[T]}^\L n}$ is idempotent over $\C[T_1,\ldots,T_n]$.
\exend
\end{rmk}

Finally, we review the definition of Berkovich spaces.

\begin{dfn}[{\cite[Section 1]{Be}}, Berkovich spectra]\label{dfn:Berkovich spectra}
Let $A$ be a Banach $\C$-algebra. The \emph{Berkovich spectrum} $\M(A)$ of $A$ is the closed subspace $\M(A)\subset\prod_{a\in A}[0,|a|_A]$ of maps $\|-\|:A\to\mathbb{R}_{\geq0}$ satisfying the following conditions:
\begin{enumerate}
    \item We have $\|0\|=0$ and $\|1\|=1$.
    \item For any element $a\in A$, we have $\|a\|\leq|a|_A$. 
    \item For any elements $a,b\in A$, we have $\|ab\|=\|a\|\|b\|$.
    \item For any elements $a,b\in A$, we have $\|a+b\|\leq\|a\|+\|b\|$.
\end{enumerate}
Moreover, we can associate any ind-Banach $\C$-algebra $\ilim[i]A_i$ to $\plim[i]\M(A_i)$.

Note that it is clearly a compact Hausdorff space by Tychonoff's theorem. Moreover, we have the notion of \emph{rational subsets}
\[
\M(A)\left(\dfrac{f_1,\ldots,f_n}{g}\right):=\{x\in\M(A)\vert\|f_i(x)\|\leq\|g(x)\|\quad\text{for all $i$}\},
\]
where $f_1,\ldots,f_n,g$ generate the unit ideal.
\exend
\end{dfn}

\begin{rmk}[Berkovich realization for complex geometry]\label{rmk:Berkovich realization for complex geometry}
Note for any (discrete) $\C$-algebra $A$, we can endow $A$ with an ind-Banach $\C$-algebra structure. In this view, the Berkovich spectrum $\M(A)$ is homeomorphic to $\Hom_{\mathrm{Alg}_\C}(A,\C)$ by Ostrowski (\cite{O}) and Gelfand--Mazur (\cite{Ge} and \cite{M}). For example, we have $\M(\C[T])=\C$ where $\C$ is the usual complex plane. Moreover, this assignment sends the (classical) Zariski covers of $\Spec(A)$ to the rational open covers of $\M(A)$, which allow us to globalize this construction to any varieties over $\C$ (i.e., $\C$-scheme of finite type). That is, for a variety $X$ over $\C$, we can consider $X(\C)$ as Berkovich spaces.
\exend
\end{rmk}

\section{The Analytic Line}\label{section:The Analytic Line}

In this section, we define the “analytic" line of complex analytic spaces.

\begin{dfn}[rings of overconvergent functions]
The ring $\C\{|T|\leq1\}$ of \emph{overconvergent functions} is defined by
\[
\bigcup_{r>1}\left\{\sum_{n=0}^\infty a_nT^n\in\C[[T]]\middle\vert|a_n|_\C=o(r^{-n})\right\}=\bigcup_{r>1}\left\{\sum_{n=0}^\infty a_nT^n\in\C[[T]]\middle\vert\sum_{n=0}^\infty|a_n|r^n<\infty\right\},
\]
where each $\left\{\sum_{n=0}^\infty a_nT^n\in\C[[T]]\middle\vert\sum_{n=0}^\infty|a_n|r^n<\infty\right\}$ (in the right-hand side) is equipped with the norm given by $\sum_{n=0}^\infty|a_n|r^n$. Note that the identification above follows from shrinking convergent radii (see also {\cite[Lemma 4.11]{BBM}}).
\exend
\end{dfn}

\begin{lem}[idempotency of $\C\{|T|\leq1\}$]\label{lem:idempotency}
The ind-Banach $\C$-algebra $\C\{|T|\leq1\}$ is idempotent in $\widehat{\D}(\C[T])$.
\end{lem}
\begin{proof}
Note that $\C\{|T|\leq1\}$ is equal to the dagger algebra 
\[
\C\left\{\dfrac{T}{1}\right\}^\dagger:=\colim[1>r]\left(\C\left\langle\dfrac{T}{r}\right\rangle\right)\simeq\colim[1>r]\left(\C\left\{\dfrac{T}{r}\right\}\right)
\]
defined in {\cite[Definition 4.12]{BBM}}. By applying \cref{rmk:idempotency}(2) to $\C[T]\xrightarrow{T\mapsto T}\C\{|T|\leq1\}$, the desired idempotency follows from {\cite[Lemma 4.14]{BBM}} (see also {\cite[Theorem 5.8]{BBM}}).
\end{proof}

We define a new ring which is related to holomorphic functions on the unit disc of the complex plane.

\begin{dfn}[ring of holomorphic functions on the unit disc]
The ind-Banach $\C$-algebra $\C\{|T|<1\}$ of \emph{holomorphic functions} is defined by
\[
\bigcap_{r<1}\left\{\sum_{n=0}^\infty a_nT^n\in\C[[T]]\middle\vert\sum_{n=0}^\infty|a_n|r^n<\infty\right\},
\]
where each $\left\{\sum_{n=0}^\infty a_nT^n\in\C[[T]]\middle\vert\sum_{n=0}^\infty|a_n|r^n<\infty\right\}$ is equipped with the norm given by $\sum_{n=0}^\infty|a_n|r^n$. Note that $\Ind(\BMod_\C)$ admits not only colimits, but also limits. It coincides with the ring of holomorphic functions on the unit open disc (in the complex plane), and the following ring
\[
\O_\C(D_1):=\lim[r<1]\left(\C\left\langle\dfrac{T}{r}\right\rangle\right)\simeq\lim[r<1]\left(\C\left\{\dfrac{T}{r}\right\}\right)
\]
discussed in {\cite[Remark 4.15]{BBM}}.
\exend
\end{dfn}

\begin{rmk}\label{rmk:strictness}
As an analogue of {\cite[Corollary 5.8]{BCV}}, we can show that a map $\bigoplus_{\ell\geq0}\C\left\{\frac{T}{r}\right\} U^\ell[1]\xrightarrow{T-U}\bigoplus_{\ell\geq0}\C\left\{\frac{T}{r}\right\} U^\ell[1]$ is a strict monomorphism. Since this map is injective, and the canonical map to cokernel (i.e. the evaluation map at $U=T$) is bounded, it suffices to show that the following statement: let $E_n=\bigoplus_{i=0}^nBU^i$ with $\|\sum_{i=0}^nb_iU^i\|_{E_n}:=\sum_{i=0}^n\|b_i\|_B$ for a Banach $\C[T]$-module $B:=\C\left\{\frac{T}{r}\right\}$ with $\|T\|=r<1$, the division map by $(T-U)$ is bounded on the kernel of the evaluation map. In fact, suppose that $\sum_{i=0}^{n+1}b_iU^i=(T-U)(\sum_{i=0}^nc_iU^i)$, then we have $c_i=-\sum_{k=i+1}^{n+1}T^{k-i-1}b_k$, which allows us to evaluate
\[
\left\|\sum_{i=0}^nc_iU^i\right\|_{E_n}\leq\sum_{i=0}^n\sum_{k=i+1}^{n+1}r^{k-i-1}\|b_k\|_B\leq\dfrac{1}{1-r}\left(\sum_{k=0}^{n+1}\|b_k\|_B\right)=\dfrac{1}{1-r}\left\|\sum_{k=0}^{n+1}b_kU^k\right\|_{E_{n+1}}.
\]
Thus, we obtain the desired boundedness.
\exend
\end{rmk}

\begin{lem}[complement]\label{lem:complement}
We let $\C\{|T|\geq1\}$ denote an ind-Banach $\C$-algebra
\[
\bigcup_{r<1}\bigcup_{m\in\Z_{\geq0}}\left\{\sum_{n=-\infty}^ma_nT^n\in\C((T^{-1}))\middle\vert\sum_{n=-\infty}^m|a_n|r^n<\infty\right\},
\]
where each $\left\{\sum_{n\leq m}a_nT^n\in\C((T^{-1}))\middle\vert\sum_{n=-\infty}^m|a_n|r^n<\infty\right\}$ is equipped with the norm given by 
\[
\max\{\max_{0\leq n\leq m}\{|a_n|\},\sum_{n=0}^\infty|a_{-n}|r^{-n}\}\}.
\]
It satisfies the following properties:
\begin{enumerate}
    \item The ind-Banach $\C$-algebra $\C\{|T|\geq1\}$ is idempotent in $\widehat{\D}(\C[T])$.
    \item The corresponding localization of $\C[T]$ is isomorphic to $\C\{|T|<1\}$.
\end{enumerate}
\end{lem}
\begin{proof}
The statement (1) follows from the formal discussion for base-change along $T\to T^{-1}$ as in {\cite[Remark 7.1]{BCV}} and \cref{lem:idempotency} (see also {\cite[Lemma 6.5]{BBM}}). Therefore, it suffices to show the statement (2). Note that the fiber of $\C[T]\to\C\{|T|\geq1\}$ is isomorphic to $(\bigcup_{r<1}\C_r^{-1}\{|T|\geq1\})[-1]$ where $\C_r^{-1}\{|T|\geq1\}$ denotes the Banach $\C$-algebra $\left\{\sum_{n=-\infty}^{-1}a_nT^n\in\C((T^{-1}))\middle\vert\sum_{n=1}^\infty|a_{-n}|r^{-n}<\infty\right\}$, we obtain
\[
\R\iHom_{\C[T]}(\fib(\C[T]\to\C\{|T|\geq1\}),\C[T])\simeq\R\plim[r<1](\R\iHom_{\C[T]}(\C_r^{-1}\{|T|\geq1\},\C[T])[1])
\]
Moreover, noting that $\C_r^{-1}\{|T|\geq1\}$ is (internally) projective and compact in $\Ind(\BMod_\C)$ since $\C_r^{-1}\{|T|\geq1\}$ admits a decomposition of the form as in the proof of \cref{lem:idempotency}, there is a further representation over $\C$:
\[
\R\iHom_\C(\C_r^{-1}\{|T|\geq1\},\C[U])[1]\simeq\iHom_\C(\C_r^{-1}\{|T|\geq1\},\C[U])[1]\simeq\bigoplus_{\ell\geq0}\iHom_\C(\C_r^{-1}\{|T|\geq1\},\C)U^\ell[1]
\]
By {\cite[Lemma 3.48, Corollary 3.50]{BBK}}, we obtain
\[
\iHom_\C(\C_r^{-1}\{|T|\geq1\},\C)\simeq\prod_{n\geq1}{}^{\leq1}(\C_{r^{-n}})^\lor\simeq\prod_{n\geq1}{}^{\leq1}(\C_{r^n}),
\]
and this representation gives an isomorphism $\iHom_\C(\C_r^{-1}\{|T|\geq1\},\C)\simeq\C\left\{\frac{T}{r}\right\}$. {\cite[Lemma 6.3]{BBK}} implies that
\[
\C\{|T|<1\}\simeq\lim[r<1]\left(\C\left\{\dfrac{T}{r}\right\}\right)\simeq\R\lim[r<1]\left(\C\left\{\dfrac{T}{r}\right\}\right),
\]
and we have $\mathrm{cone}\left(\bigoplus_{\ell\geq0}\C\left\{\frac{T}{r}\right\} U^\ell[1]\xrightarrow{T-U}\bigoplus_{\ell\geq0}\C\left\{\frac{T}{r}\right\} U^\ell[1]\right)\simeq\C\left\{\frac{T}{r}\right\}[1]$ via a similar method as in the proof of {\cite[Proposition 7.2]{BCV}}. By \cref{rmk:strictness}, $\R\iHom_{\C[T]}(\C_r^{-1}\{|T|\geq1\},\C[T])[1]$ is isomorphic to the cone of $\times(T-U)$ above. By \cref{const:structure sheaves}, the computations give the following isomorphisms:
\[
\R\iHom_{\C[T]}(\fib(\C[T]\to\C\{|T|\geq1\}),\C[T])\simeq\R\lim[r<1](\R\iHom_{\C[T]}(\C_r^{-1}\{|T|\geq1\},\C[T])[1])\simeq\R\lim[r<1]\left(\C\left\{\dfrac{T}{r}\right\}\right)\simeq\C\{|T|<1\}.
\]
\end{proof}

\begin{lem}[union]\label{lem:union}
There is the following fiber sequence in $\widehat{\D}(\C[T])$:
\[
\C[T]\to\C\{|T|\leq1\}\oplus\C\{|T|\geq1\}\to\C\{|T|\leq1\}\widehat{\otimes}_{\C[T]}^\L\C\{|T|\geq1\}.
\]
\end{lem}
\begin{proof}
This strategy essentially follows from {\cite[Proposition 5.7]{CC}}. A routine computation implies 
\[
\C\{|T|\leq1\}\widehat{\otimes}_{\C[T]}^\L\C\{|T|\geq1\}\simeq\bigcup_{r>1,s<1}\left\{\sum_{n=-\infty}^\infty a_nT^n\in\C((T))\middle\vert |a_{-n}|_\C=o(s^n), |a_n|_\C=o(r^{-n})\right\}.
\]
Dividing the nonnegative part and the negative part $s\left(\sum_{n=-\infty}^\infty a_nT^n\right)=\left(\sum_{n=0}^\infty a_nT^n,-\sum_{n=1}^\infty a_{-n}T^{-n}\right)$ for an element of $\C\{|T|\leq1\}\widehat{\otimes}_{\C[T]}^\L\C\{|T|\geq1\}$ gives an element of $\C\{|T|\leq1\}\oplus\C\{|T|\geq1\}$. Therefore, the canonical map $\C\{|T|\leq1\}\oplus\C\{|T|\geq1\}\to\C\{|T|\leq1\}\widehat{\otimes}_{\C[T]}^\L\C\{|T|\geq1\}$ is surjective. Also, by taking any element $(f,g)$ of its kernel, we obtain $f=g$, which implies $f=g\in\C[T]$ by definition of $\C\{|T|\leq1\}$ and $\C\{|T|\geq1\}$. Finally, the induced norm of $s\left(\sum_{n=-\infty}^\infty a_nT^n\right)$ is bounded by the norms of $\sum_{n=0}^\infty a_nT^n$ in $\C\{|T|\leq1\}$ and $\sum_{n=1}^\infty a_{-n}T^{-n}$ in $\C\{|T|\geq1\}$, and they are bounded by the norm of $\sum_{n=-\infty}^\infty a_nT^n$ in $\bigcup_{r>1,s<1}\left\{\sum_{n=-\infty}^\infty a_nT^n\in\C((T))\middle\vert |a_{-n}|_\C=o(s^n), |a_n|_\C=o(r^{-n})\right\}$, which implies that $s$ is a map in $\widehat{\D}(\C[T])$. This observation shows that the sequence $\C[T]\to\C\{|T|\leq1\}\oplus\C\{|T|\geq1\}\to\C\{|T|\leq1\}\widehat{\otimes}_{\C[T]}^\L\C\{|T|\geq1\}$ above can be realized as a fiber sequence in $\widehat{\D}(\C[T])$.
\end{proof}

Finally, using \cref{section:Abstract GAGA}, we obtain the following abstract GAGA.

\begin{pro}\label{pro:some relations}
For any positive number $r>0$, we let $\{|T|<1\}$ (resp. $\{|T|<r\}$) denote the open subset determined by $\C\{|T|<1\}$ of $\S(\C[T])$ in the sense of \cref{dfn:open/closed inclusions} (resp. the open subset $\{|T/r|<1\}$), and define $\{|T|\leq r\}, \{|T|>r\}, \{|T|\geq r\}$ similarly. For any element $f\in\C[T]$ and any positive number $r>0$, we let $\{|f|< r\}$ denote the preimage of $\{|T|<r\}$ induced by $\C[T]\xrightarrow{T\mapsto f}\C[T]$, and define $\{|f|\leq r\}, \{|f|>r\}, \{|f|\geq r\}$ similarly. For any elements $f,g\in\C[T]$, $\alpha\in\C$ and any positive numbers $r,s>0$, we have the following relations:
\begin{enumerate}
    \item We have $\{|f|\leq1\}=\bigcap_{r>1}\{|f|\leq r\}$ and $\{|f|\geq1\}=\bigcap_{0<r<1}\{|f|\geq r\}$.
    \item We have $\{|f|\leq1\}\cup\{|f|\geq1\}=\S(\C[T])$.
    \item If $r<1$, then we have $\{|f|\leq r\}\cap\{|f|\geq1\}=\emptyset$.
    \item We have $\{|f|\leq 1\}\cap\{|g|\leq 1\}\subset \{|fg|\leq 1\}$ and $\{|f|\geq 1\}\cap\{|g|\geq 1\}\subset \{|fg|\geq 1\}$.
    \item If $|\alpha|_\C\leq1$, then we have $\{|\alpha|\leq1\}=\S(\C[T])$. If $|\alpha|_\C\geq1$, then we have $\{|\alpha|\geq1\}=\S(\C[T])$.
    \item We have $\{|f|\leq r\}\cap\{|g|\leq s\}\subset \{|f+g|\leq r+s\}$.
\end{enumerate}
\end{pro}
\begin{proof}
The three lemmata above (\cref{lem:idempotency}, \cref{lem:complement} and \cref{lem:union}) allow us to give the similar argument as in {\cite[Proposition 5.7]{CC}}, which gives the desired relations.
\end{proof}

\begin{thm}[GAGA]\label{thm:GAGA}
The triple $(\C,\widehat{\D}(\C),\bigcup_{r>0}\{|T|<r\})$ is a GAGA setup. In particular, for a proper variety $X$ over $\C$, we obtain an equivalence $C^\an(X)\simeq C^\alg(X)$.
\end{thm}
\begin{proof}
By \cref{section:Abstract GAGA}, it suffices to show the first statement. The first statement follows from \cref{pro:some relations}.
\end{proof}

\appendix

\section{Abstract GAGA}\label{section:Abstract GAGA}

In this appendix, we summarize the machinery for GAGA briefly following {\cite[Lecture V\hspace{-1pt}I]{CC}}.

\begin{nota}
We fix the following objects:
\begin{enumerate}
    \item A Noetherian base ring $R$.
    \item A $R$-linear object $C$ of $\mathrm{Sym}$.
    \item An open subset $(\A^1)^\an\subset\S(\Mod_{R[T]}(C))$ as locales (discussed in \cref{dfn:open/closed inclusions}).
\end{enumerate}
For an $R$-algebra $A$, we let $\S(A)$ denote a locale $\S(\Mod_A(C))$, moreover for an element $f\in A$, we let $\S(A,f)$ denote the preimage of $(\A^1)^\an$ via $T\longmapsto f$.
\exend
\end{nota}

We set the following assumption for “analytification" considering $\S(A,f)\subset\S(A)$ as the subset on which $f$ is analytic.

\begin{dfn}[{\cite[Lecture V\hspace{-1pt}I]{CC}}, GAGA setup]
A triple $(R,C,(\A^1)^\an)$ is a \emph{GAGA setup} if it satisfies the following properties:
\begin{enumerate}
    \item \emph{(constants are analytic)} For any element $f\in R$, we have $\S(R,f)=\S(R)$.
    \item \emph{(the sum and product of analytic functions are analytic)} For any $R$-algebra $A$ of finite type, and any elements $f,g\in A$, $\S(A,f)\cap\S(A,g)$ is contained in $\S(A,fg)$ and $\S(A,f+g)$ both.
    \item \emph{(inverting in the analytic sense)} For any $R$-algebra $A$ of finite type, and any element $f\in A$, the open subset $\S(A[1/f],1/f)$ of $\S(A[1/f])$ is even open in $\S(A)$. We let $D^\an(f)$ denote the open subset $\S(A[1/f],1/f)$ of $\S(A)$.
    \item \emph{(analytic decomposition)} For any $R$-algebra $A$ of finite type, and any element $f\in A$, we have $\S(A)=\S(A,f)\bigcup D^\an(f)$.
    \item For any $R$-algebra $A$ of finite type, and any elements $f,g\in A$, $D^\an(f+g)$ is contained in $D^\an(f)\bigcup D^\an(g)$.
\end{enumerate}
By {\cite[Lemma 6.11]{CC}}, for any $R$-algebra $A$ and any integral closure $A^+$ of a finitely generated $R$-subalgebra of $A$, we let $\S(A,A^+)$ denote the open subset of $\S(A)$ defined by the intersection of $\S(A,f_i)$'s for any (finite) generators $\{f_i\}$.
\exend
\end{dfn}

\begin{rmk}[structure sheaves]\label{rmk:structure sheaves}
There is a canonical structure sheaf $\O$ (resp. $\O^\an$) defined by a structure sheaf in the sense of \cref{const:structure sheaves} on $\S(R[T])$ (resp. $(\A^1)^\an$). This construction also defines some sheaves $\O_A$ (resp. $\O_A^+$, resp. $\O_A^\an$) on $\S(A)$ (resp. $\S(A,A^+)$, resp. $\S(A,A)$) for any $R$-algebra $A$. Furthermore, if we work in categorified locales, the same construction holds (see \cref{dfn:categorical locales}).
\exend
\end{rmk}

We also recall the notion of discrete Huber pairs.

\begin{dfn}[discrete Huber pairs]
Let $(A,A^+)$ be a \emph{discrete Huber pair} (i.e., a pair where $A$ is a ring and $A^+$ is integrally closed subring of $A$). We define the following notions:
\begin{enumerate}
    \item ({\cite[Section 2, p.461]{H}}) For a totally ordered abelian group $\Gamma$, a \emph{valuation} of $A$ with values in $\Gamma\cup\{0\}$ is a map $v:A\to\Gamma\cup\{0\}$ satisfying the following conditions:
    \begin{itemize}
        \item For any elements $a,b\in A$, we have $v(a+b)\leq\max\{v(a),v(b)\}$.
        \item For any elements $a,b\in A$, we have $v(ab)=v(a)v(b)$.
        \item We have $v(0)=0$ and $v(1)=1$.
    \end{itemize}
    We let $\mathrm{Val}(A)$ denote the set of valuations of $A$.
    \item ({\cite[Section 2, p.461]{H}}) Two valuations $v,w$ of $A$ are \emph{equivalent} if they satisfy the following equivalent conditions:
    \begin{itemize}
        \item There is an isomorphism $\Gamma_v\cup\{0\}\xrightarrow{\sim}\Gamma_w\cup\{0\}$ of ordered monoids which is compatible with $v,w$.
        \item The support of $v$ and the valuation ring of $v$ coincide with those of $w$.
        \item For any elements $a,b\in A$, $v(a)\geq v(b)$ is equivalent to $w(a)\geq w(b)$.
    \end{itemize}
    We write $v\sim w$ when $v,w$ are equivalent.
    \item ({\cite[Section 3, p.465,467]{H}}) The set $\Spa(A,A^+)$ is defined as
    \[
    \{v\in\mathrm{Val}(A)\vert v(A^+)\leq1\}/\sim.
    \]
    Moreover, we equip $\Spa(A,A^+)$ with the topology with quasi-compact open basis given by the \emph{rational open subsets}
    \[
    U\left(\dfrac{f_1,\ldots,f_n}{g}\right):=\{v\in\Spa(A,A^+)\vert v(f_i)\leq v(g)\neq0\quad\text{for all $i$}\}.
    \]
    Note that $U\left(\dfrac{f_1,\ldots,f_n}{g}\right)$ is homeomorphic to $\Spa(A[1/g],\widetilde{A^+[f_1/g,\ldots,f_n/g]})$ where $\widetilde{A^+[f_1/g,\ldots,f_n/g]}$ is the integral closure of $A^+[f_1/g,\ldots,f_n/g]$ in $A[1/g]$.
\end{enumerate}
By {\cite[Lemma 6.12]{CC}}, the poset of rational localizations is equivalent to the opposite of the category of \emph{finitary localizations} in the sense of {\cite[Lecture V\hspace{-1pt}I]{CC}}, whose Grothendieck topology is generated by covers of the following type:
\begin{enumerate}
    \item For a finitary localization $(B,B^+)$ of $(A,A^+)$, and an element $f\in B$, there are two covers $(B,B^+)\to(B,\widetilde{B^+[f]}),(B[1/f],\widetilde{B^+[1/f]})$.
    \item For a finitary localization $(B,B^+)$ of $(A,A^+)$, and elements $f_1,\ldots,f_n\in B$ generating the unit ideal, there is a cover $\{(B,B^+)\to(B[1/f_i],\widetilde{B^+[1/f_i]})\}_{i=1}^n$.
\end{enumerate}
Note that the assignment $A\longmapsto(A,A^+)$ relates the (classical) Zariski covers of $A$ with the rational open covers of $\Spa(A,A^+)$.
\exend
\end{dfn}

To summarize the two notions, noting \cref{rmk:open/closed descent}, we can define the derived categories satisfying descent.

\begin{pro}[{\cite[Proposition 6.13 and Corollary 6.14]{CC}}, descent]
If our triple $(R,C,(\A^1)^\an)$ is a GAGA setup, then the functor $(A,A^+)\longmapsto C(A,A^+)$ from discrete Huber pairs of finite type over $(R,R)$ to $\mathrm{Sym}$ defines a map of sites from the topology given by finitary localizations to the topology given by open covers in the sense of \cref{dfn:open/closed inclusions}, where $C(A,A^+)$ denotes the corresponding localizations of $C(A)$ by $\S(A,A^+)\subset\S(A)$.

In particular, for any discrete Huber pair $(A,A^+)$ of finite type over $(R,R)$, we can define a unique sheaf $C(-)$ on $\Spa(A,A^+)$ with values in $\mathrm{Sym}$, satisfying
\[
C\left(U\left(\dfrac{f_1,\ldots,f_n}{g}\right)\right)=C\left(A\left[\dfrac{1}{g}\right],\widetilde{A^+\left[\dfrac{f_1}{g},\ldots,\dfrac{f_n}{g}\right]}\right).
\]
Moreover, we can globalize this construction to $R$-schemes of finite type.
\end{pro}

\begin{dfn}[analytic/algebraic derived category]
Assume that our triple $(R,C,(\A^1)^\an)$ is a GAGA setup. Let $A$ be an $R$-algebra of finite type. We let $C^\an(A)$ (resp. $C^\alg(A)$) denote $C(A,A)$ (resp. $C(A)$). Also, let $X$ be an $R$-scheme of finite type, we define $C^\an(X)$ and $C^\alg(X)$ similarly.
\exend
\end{dfn}

Finally, by summarizing the discussion above, we can obtain the statement of abstract GAGA.

\begin{thm}[{\cite[Theorem 7.3]{CC}}, abstract GAGA]\label{thm:abstract GAGA}
Assume that our triple $(R,C,(\A^1)^\an)$ is a GAGA setup. For a separated $R$-scheme $X$ of finite type, there is an open immersion $(X^\ad,C^\an(X))\hookrightarrow(X^{\ad/R},C^\alg(X))$ of categorified locales. Moreover, if $X$ is proper, then this map is an isomorphism, in particular, we have an equivalence $C^\an(X)\xrightarrow{\sim}C^\alg(X)$ of presentably closed symmetric monoidal $C$-linear stable $\infty$-categories.
\end{thm}
\begin{proof}
By construction, the subset $\{|f|\ll1\}$ in $\Spa(A,R)'$ means $\{|f|<\infty\}$ in $\S(A)$. Therefore, $\S(A,A^+)$ is the preimage of $\Spa(A,A^+)'\subset\Spa(A,R)'$, which implies the first statement. Via Zariski descent, the second statement holds. Finally, the third statement follows from the valuative criterion for properness.
\end{proof}

\begin{rmk}[realization for the algebraic sides]\label{rmk:realization for the algebraic side}
Note that there is a canonical isomorphism $\S(\D(A))\xrightarrow{\sim}\Spec(A)^\op$ where $\Spec(A)^\op$ denotes the topological space $\Spec(A)$ equipped with the constructible topology given by the order of specializations reversed, for all $R$-algebras $A$ of finite type (see {\cite[Lecture V, Exercise 2]{CC}} or {\cite{N}}). Therefore, assuming that we are in a GAGA setup $(R,C,(\A^1)^\an)$, via globalizing, we obtain a map $(X^\op,C^\alg(A))\to(X^{\ad/R},C^\alg(A))$ of categorified locales for all $R$-schemes of finite type.
\exend
\end{rmk}

\end{document}